\documentclass[graybox]{svmult}

\usepackage{helvet}         
\usepackage{courier}        
\usepackage{type1cm}        
\usepackage{makeidx}         
\usepackage{graphicx}        
\usepackage{multicol}        
\usepackage[bottom]{footmisc}

\usepackage{amsfonts} 
\usepackage{amsmath} 
\usepackage{amssymb}

\def\KK{{\mathbb K}}
\def\lto{\longmapsto}
\def\lra{\longrightarrow}
\def\id{\mathrm{id}} 
\def\sigo{\sigma}
\def\sigol{\stackrel{\leftarrow}\sigma}
\def\sigor{\stackrel{\rightarrow}\sigma}
\def\sol{\stackrel{\leftarrow}s}
\def\sor{\stackrel{\rightarrow}s}

  \newcounter{zlist}

\begin{document}

\title*{Homothetic Rota-Baxter systems and Dyck$^m$-algebras}
\author{Tomasz Brzezi\'nski}
\institute{Tomasz Brzezi\'nski  \at Department of Mathematics, Swansea University, 
Swansea University Bay Campus,
Fabian Way,
Swansea,
  Swansea SA1 8EN, U.K., \email{T.Brzezinski@swansea.ac.uk} 
  \and 
Faculty of Mathematics, University of Bia{\l}ystok, K.\ Cio{\l}kowskiego  1M,
15-245 Bia\-{\l}ys\-tok, Poland, 
\email{T.Brzezinski@uwb.edu.pl}  }
%
%
\maketitle

\abstract{It is shown that generalized Rota-Baxter operators introduced in [W.A.\ Martinez, E.G.\ Reyes \& M.\ Ronco, {\em Int.\ J.\ Geom.\ Meth.\ Mod.\ Phys.} {\bf 18} (2021) 2150176] are a special case of Rota-Baxter systems [T.\ Brzezi\'nski, {\em J.\ Algebra} {\bf 460} (2016), 1--25]. The latter are enriched by homothetisms and then shown to give examples of Dyck$^m$-algebras.}

\section{Introduction}
\label{sec.int}
Rota-Baxter operators first appeared in \cite{Bax:ana} in analysis of differential operators on commutative Banach algebras, then brought to combinatorics \cite{Rot:Bax}, and are now intensively studied e.g.\ in probability, renormalization of quantum field theories, the theory of operads, dialgebras, trialgebras, dendriform algebras, pre-Lie algebras etc.; see \cite{Guo:int} for  progress up to the early 2010s.

A {\em Rota-Baxter operator of weight $\lambda$} on an associative algebra $A$ over a field $\KK$ is a linear operator $R:A\to A$, such that, for all $a,b\in A$,
\begin{equation}\label{RB.lambda}
R(a)R(b) = R\left(R(a) b + a R(b) +\lambda\ ab\right),
\end{equation}
where $\lambda$ is a scalar.
The pair $(A,R)$ is often referred to as  a {\em Rota-Baxter algebra}.

In a recently published article \cite{MarRey:gen} Martinez, Reyes \& Ronco introduced a generalization of Rota-Baxter operators that involves a pair of scalars $(\alpha,\beta)$. A {\em generalized Rota-Baxter operator of weights $(\alpha,\beta)$} on an associative algebra $A$ is a linear map $\bar{R}:A\to A$, such that, for all $a,b\in A$,
\begin{equation}\label{RB.al.bet}
\bar{R}(a)\bar{R}(b) = \bar{R}\left(\bar{R}(a) b + a \bar{R}(b) +\alpha\ ab\right) +\beta ab. 
\end{equation}
The authors of \cite{MarRey:gen} then proceed to show that such operators of weights $(3,2)$  give rise to a class of Dyck$^m$-algebras introduced and studied in \cite{LopPre:alg}, \cite{LopPre:sim} in context of dendriform structures. 

This note has two aims. First, we show that generalized Rota-Baxter operators are examples of Rota-Baxter systems \cite{Brz:Rot}. Second, we enrich Rota-Baxter systems with homothetisms \cite{Red:ver} or self-permutable bimultiplications \cite{Mac:ext} that were introduced in studies of ring extensions and constrain them in such a way as to produce examples of Dyck$^m$-algebras that  extend those in \cite{MarRey:gen}.

\section{Rota-Baxter systems, homothetisms and Dyck$^m$-algebras}
\label{sec.RB}

\subsection{Rota-Baxter systems and generalized Rota-Baxter operators}\label{ssec.RB}
The following notion was introduced in \cite{Brz:Rot}. An associative algebra $A$ (over a field $\KK$) together with a pair of linear operators $R,S: A\lra A$ is called a {\em Rota-Baxter system} if, for all $a,b\in A$,
\begin{subequations}\label{RB}
\begin{gather}
R(a)R(b) = R\left(R(a) b + a S(b)\right) , \label{RB.r}\\
S(a)S(b) = S\left(R(a) b + a S(b)\right) . \label{RB.s}
\end{gather}
\end{subequations}
As explained in \cite{Das:gen}, conditions \eqref{RB} can be recast in the form of a single Nijenhuis operator. Recall from \cite{CarGra:qua} that a Nijenhuis tensor or operator on an associative algebra $B$ is a linear function $N: B\lra B$ such that, for all $a,b\in B$,
\begin{equation}\label{Nij}
N(a)N(b) = N(N(a)b + aN(b) - N(ab)).
\end{equation}
Starting with an algebra $A$ we can form an algebra $B$ on the vector space $A\oplus A\oplus A$ with the product
$$
\begin{pmatrix}
a\cr b \cr c
\end{pmatrix}
\begin{pmatrix}
a'\cr b' \cr c'
\end{pmatrix}
= 
\begin{pmatrix}
ab\cr a'b' \cr ac'+cb'
\end{pmatrix}
$$
Then $(A,R,S)$ is a Rota-Baxter system if and only if 
$$
N = \begin{pmatrix}
0 & 0 & R\\
0 & 0 & S\\
0 & 0  &0
\end{pmatrix}
$$
is a Nijenhuis operator on $B$.

If $R$ is a  Rota-Baxter operator of weight $\lambda$, then setting $S=R+\lambda\ \id$ one obtains a Rota-Baxter system. In a similar way, one can interpret a generalized Rota-Baxter operator of weights $(\alpha,\beta)$ as a Rota-Baxter system. More precisely,

\begin{lemma}\label{lem.RB}
Let $\bar{R}$ be a generalized Rota-Baxter operator of weights $(\alpha,\beta)$ and let $\lambda,\mu\in \KK$ be such that
\begin{equation}\label{ablm}
\alpha = \lambda+ \mu, \qquad \beta = \lambda\mu.
\end{equation}
Set
\begin{equation}\label{abrs}
R = \bar{R} +\lambda\ \id, \qquad S = \bar{R} +\mu\ \id.
\end{equation}
Then $(A,R,S)$ is a Rota-Baxter system.
\end{lemma}
This is checked by a straightforward calculation (left to the reader). As a consequence and in view of \cite[Proposition~2.5]{Brz:Rot}, one concludes, for example that there is a dendriform algebra associated to a generalized Rota-Baxter operator $\bar{R}$ (cf.\ \cite[Proposition~18]{MarRey:gen}).

Of course, if $\KK$ is not an algebraically closed field, equations \eqref{ablm} for $\lambda$ and $\mu$ might not have solutions. However, the most interesting case of a generalized Rota-Baxter operator studied in \cite{MarRey:gen} corresponds to the weights $(3,2)$, hence one can take $\lambda =1$ and $\mu=2$ then.

\subsection{Homothetisms and homothetic Rota-Baxter systems}\label{ssec.hom}
The following notion originates from studies of homology of rings and ring extensions in 
 \cite{Red:ver} and \cite{Mac:ext}. Let $A$ be an associative algebra. By a {\em double operator} $\sigma$ on $A$ we mean a pair of linear operators   $\sigma = (\sigor,\sigol)$  on $A$,
$$
\sigor:A\lra A,\quad a\lto \sigo a,\qquad \sigol:A\lra A,\quad a\lto a\sigo .
$$
The somewhat unusual way of writing the argument to the left of the operator (in the definition of $\sigol$) proves very practical and economical in expressing  the action of double operators with additional properties, in particular those that we are introducing presently.

A double operator $\sigo$ on $A$ is called a  {\em bimultiplication} \cite{Mac:ext} or a {\em bitranslation} \cite{Pet:ide} if, for all $a,b\in A$,
\begin{equation} \label{h.linear}
\sigo(ab) = (\sigo a)b, \qquad (ab)\sigo = a( b\sigo)\quad \& \quad
a(\sigo b) = (a\sigo) b.
\end{equation}

A bimultiplication $\sigo$ is called a
{\em double homothetism} \cite{Red:ver} or is said to be  {\em self-permutable} \cite{Mac:ext} provided that, for all $a\in A$,
\begin{equation}\label{h.comm}
(\sigo a) \sigo =\sigo (a \sigo) .
\end{equation}
The first two conditions in \eqref{h.linear} mean that $\sigor$ is a right and $\sigol$ is a left $A$-module homomorphism. A bimultiplication is called simply a   {\em multiplication} in  \cite{Hoc:coh}. In functional analysis, in particular in the context of $C^*$-algebras, bimultiplications are known as {\em multipliers} \cite{Hel:mul}, \cite{Bus:dou}. The set of all bimultiplications is a unital algebra, known as a {\em multiplier algebra}. The relations  \eqref{h.comm} mean that  endomorphisms $\sigor$ and  $\sigol$ mutually commute in the endomorphism algebra of the vector space $A$.  Put together, conditions \eqref{h.linear} and \eqref{h.comm} mean that one needs not put any brackets in strings of letters that involve elements of $A$ and $\sigma$.  

Any element of $A$, say $s\in A$, induces a double homothetism $\bar{s} = (\sor,\sol)$ on $A$,
$$
\sor: a\lto sa, \qquad \sol: a\lto as.
$$
Such double homothetisms are said to be {\em inner} and they form an ideal in the multiplier algebra. In applications of bimultiplications to ring extensions, most recently in connecting extensions of integers to trusses (sets with an associative binary operation distributing over a ternary abelian heap operation) \cite{AndBrz:ide} the key role is played by the quotient of the multiplier algebra by the ideal of inner homothetisms; in particular in the theory of operator algebras this is known as the  {\em corona algebra}. 

The rescaling by $\lambda\in \KK$ understood as a pair $(\lambda \id, \lambda\id)$ is a double homothetism, which we will also denote by $\lambda$. If $A$ has the identity $1$, then, of course $\lambda$ is an inner homothetism, $\overline{\lambda 1}$.  

We are now ready to define the main notion of this note.
\begin{definition}
Let $(A,R,S)$ be a Rota-Baxter system and let $\sigma$ be a double homothetism  on $A$. We say that $(A,R,S,\sigma)$  is a {\em homothetic Rota-Baxter system} if, for all $a\in A$,
\begin{equation}\label{h.cons}
S(a)\sigma - \sigma R(a) = \sigma a \sigma.
\end{equation}
\end{definition}

If $\bar{R}$ is a generalized Rota-Baxter operator of weights $(\alpha,\beta)$ such that 
$$
\gamma := \pm \sqrt{\alpha^2 - 4\beta} \in \KK,
$$
then $(A,R,S, \gamma)$ is a homothetic Rota-Baxter system, where $(A,R,S)$ corresponds to $\bar{R}$ through Lemma~\ref{lem.RB} (the sign depends on the choice of $\lambda$ and $\mu$, i.e.\ whether $\lambda < \mu$ or $\lambda >\mu$). 

\subsection{Dyck$^m$-algebras}\label{ssec.D}
The combinatorial, algebraic and operadic aspects of a certain class of lattice paths counted by Fuss-Catalan numbers led  L\'opez, Pr\'eville-Ratelle \& Ronco in \cite{LopPre:alg} to introduce the following notion.

Let $m$ be a natural number. A $\KK$-vector space $A$ together with $m+1$ linear operations $\ast_i: A\otimes A\lra A$, $i=0,\ldots , m$ such that, for all $a,b,c\in A$
\begin{subequations}\label{Dyck}
\begin{equation}\label{Dyck.1}
a\ast_i (b\ast_j c) = (a\ast_i b)\ast_j c, \qquad 0\leq i<j\leq m,
\end{equation}
\begin{equation}\label{Dyck.2}
a\ast_0(b\ast_0 c) = \left(\sum_{i=0}^m a\ast_i b\right)\ast_0 c, 
\end{equation}
\begin{equation}\label{Dyck.3}
a\ast_m \left(\sum_{i=0}^m b\ast_i c\right) = (a\ast_m b)\ast_m c, 
\end{equation}
\begin{equation}\label{Dyck.4}
a\ast_i \left(\sum_{k=0}^i b\ast_k c\right) = \left(\sum_{k=i}^m a\ast_k b\right)\ast_i c, \qquad 1\leq i\leq m-1,
\end{equation}
\end{subequations}
is called a {\em Dyck$^m$-algebra}.

 Dyck$^m$-algebras generalize associative algebras (the $m=0$ case) and Loday's dendriform algebras (the $m=1$ case) \cite[Section~5]{Lod:dia}. In \cite{MarRey:gen} it is shown that one can associate Dyck$^m$-algebras to any generalized Rota-Baxter operator of weights $(3,2)$ (see Theorem~20 and Proposition~21 in \cite{MarRey:gen} for explicit formulae). Aided by this observation we will associate Dyck$^m$-algebras  to any homothetic Rota-Baxter system.


\section{Dyck$^m$-algebras from homothetic Rota-Baxter systems}\label{sec.main}
The main result of this note is contained in the following theorem.

\begin{theorem}\label{thm.main}
Let $(A,R,S,\sigma)$ be a homothetic Rota-Baxter system and let $m$ be any natural number. Define $m+1$ linear operations $\ast_i: A\otimes A \lra A$, $i=0,\ldots, m$ as follows:
\begin{equation}
a\ast_i b = \begin{cases} 
R(a) b, & i=0,\\
(-1)^{i+1}a\sigma b, & i = 1,\ldots, m-1,\\
a S(b) -\frac{1 + (-1)^m}2 a\sigma b, & i=m,
\end{cases}
\end{equation}
for all $a,b\in A$. Then $\left(A, \{\ast_i\}_{i=0}^m\right)$ is a Dyck$^m$-algebra.
\end{theorem}
\begin{proof}
We will carefully check that all the relations between different operations listed in  \eqref{Dyck} hold. Starting with \eqref{Dyck.1}, if $i\neq 0$ and $j\neq m$, these equations reduce to the equality $a\sigma(b\sigma c) = (a\sigma b)\sigma c$ which follows by \eqref{h.linear} and \eqref{h.comm}. For $i=0$ and $j\neq m$, \eqref{Dyck.1} amounts to the equality $R(a)(b\sigma c) = (R(a)b)\sigma c$, which holds by \eqref{h.linear}. The proofs of \eqref{Dyck.1} for $j=m$, and all the remaining equalities in \eqref{Dyck} depend on the parity of $m$. So, we will consider two separate cases in turn.

Assume that $m$ is odd. Then \eqref{Dyck.1} with $i=0$ and $j=m$ follows immediately by the associativity of $A$, while $i\neq 0$ and $j=m$ amounts to the equality $a(b\sigma S(c)) = (ab)\sigma S(c)$, which holds by \eqref{h.linear}.

Since $m$ is odd \eqref{Dyck.2}  reduces to 
$$
a\ast_0(b\ast_0 c) = (a\ast_0b +a\ast_mb )\ast_0 c,
$$
that is 
$$
R(a)R(b)c = R(R(a)b+aS(b))c,
$$
and this immediately follows by \eqref{RB.r}. In a similar way \eqref{Dyck.3} follows by \eqref{RB.s}. 

We split checking \eqref{Dyck.4} into two cases. If $i$ is even, then \eqref{Dyck.4} reduces to
$$
a\ast_i(b\ast_0c) = (a\ast_m b + a\ast_{m-1} b)\ast_i c,
$$
which amounts to the equality
\begin{equation}\label{i.even}
a\sigma R(b) c = (aS(b)-a\sigma b)\sigma c,
\end{equation}
that follows by the constraint \eqref{h.cons}. If $i$ is odd, then  \eqref{Dyck.4} is equivalent to 
$$
a\ast_i(b\ast_0c+ b\ast_1c) = (a\ast_m b )\ast_i c,
$$
that is,
\begin{equation}\label{i.odd} 
R(a)(R(b)c +b\sigma c) = (aS(b))\sigma c
\end{equation}
and thus again follows by the definition of a double homothetism  and \eqref{h.cons}. This completes the proof of the theorem for $m$ odd.

Assume now that $m$ is even. We look back at two remaining cases in \eqref{Dyck.1}. If $i=0$ and $j=m$, then
$$
\begin{aligned}
a\ast_0(b\ast_m c) &= R(a)bS(c) - R(a)(b\sigma c)\\
& = R(a)bS(c) - (R(a)b)\sigma c = (a\ast_0b)\ast_m c,
\end{aligned}
$$
by the associativity of $A$ and \eqref{h.linear}. In a similar way, if $i\neq 0$ and $j=m$, \eqref{Dyck.1} is equivalent to the equality
$$
a\sigma(bS(c)) - a\sigma (b\sigma c) = (a\sigma b)S(c) - (a\sigma b)\sigma c,
$$
which follows by \eqref{h.linear} and \eqref{h.comm}.

Since $m$ is even \eqref{Dyck.2}  reduces to 
$$
a\ast_0(b\ast_0 c) = (a\ast_0b +a\ast_1 b + a\ast_mb )\ast_0 c,
$$
and, in view of the definition of $\ast_m$,  this immediately follows by \eqref{RB.r}. In a similar way \eqref{Dyck.3} follows by \eqref{RB.s}. 

As for the $m$-odd case, we split checking \eqref{Dyck.4} into two cases. If $i$ is even, then \eqref{Dyck.4} reduces to
$$
a\ast_i(b\ast_0c) = (a\ast_m b)\ast_i c,
$$
which is the same as \eqref{i.even}, while, for $i$ odd,  \eqref{Dyck.4} is equivalent to 
$$
a\ast_i(b\ast_0c+ b\ast_1c) = (a\ast_m b+a\ast_{m-1} b )\ast_i c,
$$
that is already proven equality \eqref{i.even}. This completes the proof of the theorem.
\end{proof}

If $\bar{R}$ is a generalized Rota-Baxter operator of weights $(3,2)$, then the corresponding Rota-Baxter system $R = \bar{R} + \id$, $S = \bar{R} +2\id$ is constrained by the homothetism $1$, induced by the rescaling by the identity in $\KK$, and thus Theorem~\ref{thm.main} implies \cite[Theorem~20 \& Proposition~21]{MarRey:gen}.
\begin{acknowledgement}
The research  is partially supported by the National Science Centre, Poland, grant no. 2019/35/B/ST1/01115.
\end{acknowledgement}
%
%


\end{document}